\documentclass[a4paper,12pt]{amsart}
\usepackage{amssymb,amsfonts,amsmath}
\usepackage{layout}
\usepackage[utf8]{inputenc}
\usepackage[T1]{fontenc}

\def\m{\mathcal}

\def\C{\mathbb{C}}
\def\c2{\mathbb{C}^2}
\def\R{\mathbb{R}}
\def\N{\mathbb{N}}

\def\N{\mathbb{N}}

\def\1{\bold{1}}

\def\a{\alpha}

\def\e{\varepsilon}
\def\f{\varphi}

\def\G{\Gamma}
\def\p{\psi}

\newcommand \W {\Omega}

\newcommand \mE {\mathcal E}
\newcommand \mF {\mathcal F}

\newcommand \mG {\mathbb G}
\newcommand \vk {\chi}
 \newcommand \Ri{ \Rightarrow }
\newcommand \Sub {\Subset}
\newcommand \sub {\subset}
\newcommand \sm {\setminus}
\newcommand \ove {\overline}

\newcommand \mrm {\mathrm}

\newtheorem{lem}{Lemma}[section]

\newtheorem{pro}[lem]{Proposition}
\theoremstyle{definition}
\newtheorem{defi}[lem]{Definition}
\theoremstyle{plain}
\newtheorem{def/not}[lem]{Definition/Notations}
\numberwithin{equation}{section}
\newtheorem{thm}[lem]{Theorem}
\newtheorem{cor}[lem]{Corollary}

\newenvironment{proof3.1}
{\noindent {\it{Proof of theorem 3.1}}}{$\Box$ \linebreak[4]}

\begin{document}

\title[Weighted Pluricomplex energy II ]
{Weighted Pluricomplex energy II  }
\author{ Slimane BENELKOURCHI }
\address{
 Université de Montréal,
%Faculté de l’éducation permanente,
Pavillon 3744, rue Jean-Brillant,
Montréal QC  H3C 3J7.
}
 
\email{slimane.benelkourchi@umontreal.ca}
%\dedicatory{Dedicated to Professor Donald Knuth on the occasion
 %    of his $100$th birthday}
\maketitle
%%%%%%%%%%%%%%%%%%%%%%%%%%%%%%%%%%%%%%%%%%%%%%%%%%%%%%%%%%Ã¹
\begin{abstract}
We continue our study of the Complex Monge-Ampère Operator on the Weighted Pluricomplex energy classes.
We give more characterizations of the range of the classes $\m E_ \chi$ by the Complex Monge-Ampère Operator. In 
particular, we prove that a non-negative Borel measure $\mu $ is the Monge-Ampère of a unique function 
$\f \in \m{ E}_\chi$ if and only if $\chi(\m E _\chi ) \sub L^1(d\mu ).$
Then we show that if $\mu = (dd^c \f )^n $ for some $\f \in \m E_\chi $ then
  $\mu = (dd^c u )^n $ for some $u \in \m E_\chi (f)$ where $f$ is a given boundary data. If moreover, the non-negative
Borel measure$\mu $ is suitably dominated by the Monge-Ampère 
capacity, we establish a priori estimates on the capacity of sub-level sets of the solutions.
As consequence, we give a priori bounds of the solution of the Dirichlet problem in the case when the
measure has a density in some Orlicz space. 
\end{abstract}
%%%%%%%%%%%%%%%%%%%%%%%%%%%%%%%%%%%%%%%%%%%%%%%%%%%%%%%%%%%%%%
\section{Introduction}%{Energy classes $\mE_\chi $}
Let $\W \sub \C^n $ be a bounded hyperconvex domain, i.e. a connected, bounded open subset of $\C ^n $ 
such that there exists a negative plurisubharmonic function $\rho $  such that 
$\{ z\in \W ; \rho (z) < -c\} \Sub \W , $ $ \forall c>0.$ Such a function $\rho $ 
is called an exhaustion function. We let $PSH(\W) $  denote the cone of plurisubharmonic functions (psh for short)
on $\W$ and $PSH^- (\W)$ denote the subclass of negative functions.
athcal
As known (see \cite{BT1} and \cite{BT2}), the complex Monge-Ampère operator $(dd^c \cdot )^n$ is well defined, as a non-negative measure, on the set of locally bounded plurisubharmonic functions.
%This operator cannot be defined for some unbounded psh functions (see \cite{Ki}).
Therefore the  question of describing the measures which are the Monge-Ampère of bounded psh functions is very 
important for pluripotential theory, complex dynamic... This problem has been studied extensively by various authors,
  see for example  \cite{BT2}, 
 \cite{K94}, \cite{K98}, \cite{K05}, \cite{Cz9}... and reference therein.
 In \cite{Ce98}, Cegrell introduced the pluricomplex energy classes $\m E_p(\W)$ and $\m F_p (\W)$ ($p \ge 1$) on
which the complex Monge–Amp\`ere operator is well defined. 
 He proved that a measure $\mu$ is the Monge-Amp\`ere of some function $u \in \m E_p (\W)$ if and only if it satisfies 
\begin{equation} \label{Ce}
\int_\W (-v)^p d \mu \le Const \left( \int _\W (-v)^p (dd^c v)^n \right) ^{p \over (n+p)} , \quad \forall v \in \m E_0(\W),
\end{equation}
where $\m E_0(\W)$ is the cone of all bounded psh functions $\f$ defined on the domain 
$\W$ with finite total Monge-Ampère mass and $\lim _{z \to \zeta } \f(z) = 0,$
for every $\zeta \in \partial \W.$ Recently, \AA hag,  Cegrell and  Czy\.z  in \cite{ACC} proved that, in the case $p=1$ the inequality (\ref{Ce})
is equivalent to $\m E_1 (\W ) ) \sub L^1 ( d\mu ).$ In this note, our first objective  is to extend this result by showing 
 that, for all positive number $p ,$ the inequality (\ref{Ce})
is equivalent to  $\m E_p (\W ) ) \sub L^p ( d\mu ).$ In fact, we prove some more general result.
Given  a non-decreasing function  $\chi : \R^-  \to \R^- ,$  we consider
 the set $\m E_\chi (\W ) $
  of plurisubharmonic functions  of finite $\chi - $weighted Monge-Ampère energy 
and, in some sense, has boundary values zero. These are the functions $u \in PSH(\Omega)$ for which
 there exists a decreasing sequence $u_j \in {\mathcal E}_0(\Omega)$ with limit 
  $u$  and 
 $$
\sup_{j \in \N} \int _ {\W}-\chi( u_j) \, (dd^c u_j )^n
<\infty.
$$
 
Then we have the following characterization of the image of the complex Monge-Ampère acting in the class $\m E _\chi (\W) $.
\begin{thm} \label{thm1}
Let $\chi : \ \R^- \to \R^- $ be
an increasing convex or homogeneous  function such that $\chi(-\infty)
 = -\infty .$
  The following assertions are equivalent:\\

(1)
 there exists  a unique function $\f \in \m E_\chi(\W)$ such
that $\mu = (dd^c \f )^n$;\\

(2) $\chi (\m E_\chi (\W ) ) \sub L^1 ( d\mu ).$
\end{thm}
 
Next, we extend our previous result to families of functions having prescribed 
boundary data. 
Let $f \in PSH(\W) $ be a maximal psh function. We define the class $\m E_\chi (f)$ to be 
the class of psh functions $u$ such that there exists a function $\f \in \m E _\chi (\W) $ 
such that 
$$
\f (z) + f(z)\le u(z) \le f(z) , \quad \forall z \in \W.
$$
Some particular cases of the classes $\m E_\chi (f) $  has been studied in \cite{Ah}, \cite{ACCP}, \cite{B09},
\cite{BGZ2}, \cite{Ce98}, \cite{Ce04}, \cite{Ce08}, \cite{Cz9}, \cite{HH}, \cite{HHNN}.

 More precisely, we prove the following result.

\begin{thm}\label{thm2}
Let $\mu $  be a non-negative measure in $\W$, $\chi : \R^- \to \R^- $ be an increasing convex or homogeneous
function such that $\chi ( -\infty ) = -\infty $ and $f$ be a maximal function.
 Then, if $\mu = (dd^c u)^n$ for some $u \in \m E_\chi(\W)$ then there exists a unique function 
$\f \in \m E_\chi (f)$  such that  $\mu = (dd^c \f)^n$.

\end{thm}
%%%%%%%%%%%%%%%%%%%%%%%%%%%%%%%%%%%%%%%%%%%%%%%%%%%%%%%%%%%%%%%%%%%%%%%%%%%%%%%%%%%%%%%%%%

Moreover, when the non-negative measure $\mu $ is dominated by the 
Monge-Ampère capacity, we give an estimate of the growth of solutions of the equation $(dd^c \f)^n  = \mu$ .
As in \cite{BGZ2}, let consider the function
$$
F_\mu (t) : = \sup \{ \mu (K), \ K \Sub \W ; \ cap_\W (K) \le t \} , \ \forall t \ge 0.
$$
Then $ F : = F_\mu $ is a non-decreasing function on $ \R^+ $ and satisfies
\begin{equation}
\mu (K) \le  F_\mu (cap_\mu (K)),  \quad \text{for all Borel subsets \ }\ K \sub \W. 
\end{equation}
Write $F(x) = F_\e (x) = x \left (\e(-\ln {x\over n})\right )^n$ where $\e : \R^+ \to \R^+ $ is non-decreasing.
\\
Such measures dominated by the Monge-Ampère capacity have extensively studied by S.Kolodziej in
\cite{K94}, \cite{K98}, \cite{K05}. He proved that if $\phi : \partial \W \to \R $ is a continuous function
and $\int ^{+\infty}_0\e(t) dt <+\infty,$ then $\mu $ is the Monge-Ampère of a unique function $\f \in PSH(\W) $
with $\f _{/\partial \W} = \phi.$

When, $\int^{+\infty }_0 \e(s) ds =+\infty,$ we have the following estimate.

\begin{thm} 
\label{thm3}
Let $\mu $ a positive finite measure.
Assume for all compact subsets $K  \subset \W$,
\begin{equation} \label{dom}
\mu(K) \leq F_{\e} \left (\mrm{Cap }_\W(K)\right ).
\end{equation}
Then there exists a unique function $\f \in \mF (f)$
such that $\mu=(dd^c \f)^n$,
and
\begin{equation*}
\mrm{Cap}_\W(\{\f<f-s  \}) \leq  \exp (-nH^{-1}(s)), \text{ for all }\ s>0,
\end{equation*}
Here $H^{-1}$ is the reciprocal function of
$H(x) =e\int_{0}^x  \e(t) dt + s_0 ( \mu )$.

In particular if  $\int^{+\infty}_0 \e (t) dt < +\infty $ then 
$$
0\le f-\f \le e\int_{0}^{+\infty}  \e(t) dt + e \e(0) + \mu( \mu )^{1\over n}.
$$
\end{thm}
% For similar results in the case $f \equiv 0$ or on the compact K\"ahler manifolds, we refer the reader
%to \cite{BGZ2}, \cite{EGZ}, \cite{GZ}.

The paper is organised as follows. In section 2, we recall the definitions of the energy classes $\m E_\chi(\W),$
and some classes of psh functions introduced by U.Cegrell \cite{Ce98}, \cite{Ce04}, \cite{Ce08} and
we prove Theorem \ref{thm1}. In section 3, we prove Theorem \ref{thm2}. As a consequence, we generalize
 the main result in the paper \cite{Ah}. In section 4, we prove Theorem \ref{thm3}. As application,
 we give a priori bound of the solution of Dirichlet problem in the case when the measure $\mu = g d\lambda$
where $g$ belongs to some Orlicz space $ L \log ^\a L .$

\section{Energy classes with zero boundary data $\m E _\chi$}
Let recall some Cegrell's classes (Cf. \cite{Ce98}, \cite{Ce04} and \cite{Ce08}).
 The class $\m E (\W)$
 is the set of plurisubharmonic functions
 $u $
such that for all $z_0 \in \Omega $,
 there exists a neighbourhood $V_{z_0}$ of
$z_0$ and $u_j \in {\mathcal E}_0(\Omega)$ a decreasing sequence
which converges towards $u$ in $V_{z_0}$ and satisfies $\sup_j
\int_{\Omega} (dd^c u_j)^n <+\infty$. U.Cegrell \cite{Ce04} has shown 
 that the operator $(dd^c \cdot )^n$ is well defined on
$\mE(\Omega),$ continuous under decreasing limits and the class
$\mE(\Omega)$ is stable under taking maximum i.e. if $u\in \mE(\W) $ 
and $v\in PSH^-(\W)$ then $\max (u , v) \in \mE(\W).$ This class is the largest
class with these properties (Theorem 4.5 in \cite{Ce04}). The class
$\mE(\Omega)$ has been further characterized by Z.Blocki \cite{Bl
1}, \cite{Bl 2} and Hai, Le Mau; Hiep, Pham Hoang; Quy, Hoang Nhat in \cite{HHQ}.

The class ${\mathcal F}(\Omega)$ is the global version of
 $\mE(\Omega)$:
a function $u$ belongs to ${\mathcal F}(\Omega)$ iff there exists
 a decreasing sequence
 $u_j \in {\mathcal E}_0(\Omega)$
converging towards $u$ {\it in all of } $\Omega$, which satisfies
$\sup_j \int_{\Omega} (dd^c u_j)^n<+\infty$. The class
$\mE(\Omega)$ has been further characterized in \cite{Bl 1} and \cite{BGZ2}.

%The class ${\mathcal F}^a(\Omega)$ is the set of functions
%%%%%%$(dd^c u)^n$
%vanishes on pluripolar sets.

 Let $\W_j\Sub \W$ be an
increasing sequence of strictly pseudoconvex domains such that $ \W=
\cup _j \W_j.$ Let $u\in \mE(\W)$ be a  given psh function and put
$$
 u_{\W _j}:= \sup \left \{\f \in PSH(\W); \ \f \le u \ \text{on}\
 \W\sm \W_j\right \}.
$$
Then we have  $u_{\W_j} \in \m E(\W)$ and $u_{\W_j}$ is an increasing sequence. Let
$\tilde{u}:= (\lim_j u_{\W_j})^* .$  It follows from the properties of $\mE(\W)$
 that $\tilde{u} \in \m E(\W).$ Note that the definition of
$\tilde{u}$  is independent of the choice of the sequence $\W_j$ and
is maximal i.e. $(dd^c \tilde{u})^n =0.$ $\tilde{u}$ is the smallest
maximal psh function above $u$. Define $\m N(\W):= \{u\in \mE(\W); \
\tilde{u} =0\}.$ In fact, this class is the analogous of potentials
for subharmonic functions.
%Also, denote $\m N^a(\W)=\m E^a(\W)\cap \m N(\W).$
\begin{defi}
 Let $\chi : \R^-  \to \R^- $ be a
 non-decreasing
 function. We let $ \m E_\chi (\W ) $
 denote the set of all functions $u \in PSH(\Omega)$ for which
 there exists a sequence $u_j \in {\mathcal E}_0(\Omega)$
 decreasing to $u$ in $\Omega$ and satisfying
 $$
\sup_{j \in \N} \int _ {\W}-\chi( u_j) \, (dd^c u_j )^n
<\infty.
$$
\end{defi}

It was proved in \cite{B11} and \cite{HH} that if
%Let $\chi : \R^-  \to \R^- $ be a nondecreasing function such that
$\chi \not \equiv 0.$ Then
 $$
\m E_\chi (\W )\sub \mE (\W).
$$
 In particular, for any function $u\in \mE_\chi(\W),$ the complex Monge-Ampère operator
 $(dd^c u)^n$ is well defined as non-negative measure.
 Furthermore, if $\chi(-t)< 0\ \forall\  t>0  ,$ then
  $$
\m E_\chi (\W )=\left \{u\in  \m N(\W); \ \int _\W-\chi(u)(dd^c u)^n <+\infty\right \}.
$$
%%%%%%%%%%%%%%%%%%%%%%%%%%%%%%%%%%%%%%%%%%%%%%%%%%%%%%%%%%%%%%%%
  The classes $\m E_\chi(\W)$ has bee characterized by the
speed of decrease of the capacity of sublevel sets  \cite{B09}, \cite{BGZ2}.\\
 Recall that the Monge-Ampère capacity has been
introduced and studied by E.Bedford and A.Taylor in \cite{BT1}.
Given $K\sub \W$ a compact subset, its Monge-Ampère capacity
relatively to $\W$ is defined by
$$
\mrm{cap}_\W (K) := \sup \left \{ \int _K(dd^c u)^n ; \ u \in
PSH(\W),
 \ -1\le u \le 0 \right \}.
$$
The following estimates (cf \cite{BGZ2}) will be useful later on. For any $\f \in \m E_0$
\begin{equation}\label{estimate}
t^n \mrm{cap}_\W ( \f < -s -t ) 
\le \int_{(\f < -s)}(dd^c \f )^n 
\le s^n \mrm{cap}_\W (\f <-s ) , \ \forall s,\ t >0.
\end{equation}

Let $\chi : \R^-  \to \R^- $ be a non-decreasing function. Without loss of generality, from now on,
 we assume that
$\chi(0) = 0.$
We define the class $\hat{{\mathcal E}}_{\chi}(\W )$
$$
\hat{{\mathcal E}}_{\chi}(\W ) :=\left\{ \f \in PSH^-(\W) \, / \,
\int_{t_\chi}^{+\infty} t^n \chi\prime(-t) \mrm {Cap}_\W(\{\f<-t\}) dt<+\infty
\right\}.
$$
\begin{pro} \label{capdef} We have
$\hat{{\mathcal E}}_{\chi} (\W ) \subset \mE _\chi
(\W ) ,$ while
$$
{\mathcal E}_{{\chi}}(\W) \subset \hat{{\mathcal E}}_{\hat{
\chi}}(\W), \text{ where }  \hat{\chi}(t) = \chi(t/2).
$$
Moreover, if $\chi : ]-\infty, -\ t_\chi [  \to \R^- $ is convex. Then
$$
{{\mathcal E}}_{\chi}(\W ) = \hat{\mathcal E}_{{\chi}}(\W).
$$
Here $t_\chi $ denote the real number satisfying $\chi(t)<0 ,\ \forall t<-t_\chi $ and 
$\chi(t)  = 0 , \ \forall t\ge -t_\chi.$
\end{pro}
\begin{proof}
Cf. \cite{B11}, \cite{BGZ2}.
\end{proof}
%%%%%%%%%%%%%%%%%%%%%%%%%%%%%%%%%%%%%%%%%%%%%%%%%%%%%%%%%%%%%%
%\section{$(dd^c \cdot )^n (\mE_\chi)$}
\begin{thm} \label{range}
Let $\chi : \ \R^- \to \R^- $ be
an increasing convex  function such that $\chi(-\infty)
 = -\infty .$
  The following conditions are equivalent:\\

(1)
 there exists  a unique function $\f \in \m E_\chi(\W)$ such
that $\mu = (dd^c \f )^n$;\\

(2) $\chi ( \m E_\chi(\W))  \sub L^1(\W, d\mu);$\\

(3) there exists a  constant
$C>0$ such that
\begin{equation}\label{quant}
\int_\W -\chi ( u) d\mu \le C , \ \forall \ u\in \tilde{\m E
_0}(\W);
\end{equation}

(4)
 there exists a  constant $A>0$ such that

\begin{equation}\label{quant1}
\int_\W -\chi ( u ) d\mu \le C_2
  \max \left (1, \left (\int_0^\infty s^n \chi^\prime (-s)
\mrm{cap}_\W (u< -s) ds \right )^{\frac{1}{n}}\right ),
 \ \forall \ u\in \m E _0(\W);
\end{equation}

(5) there exists a  locally bounded function $F : \R^+ \to \R^+$ such that
$\limsup_{t\to +\infty } F(t)/t <1$ %$F(0)=0$
and
\begin{equation}\label{quant2}
\int_\W -\chi(u) d\mu \le F  \left ( C_\chi(u) \right ),
 \ \forall \ u\in \m E _0(\W).
\end{equation}
Here $\tilde{\m E _0}(\W)$ denotes the class  $\tilde{\m E _0}(\W): = \{u\in {\m E _0} (\W);  \int_0^\infty s^n \chi^\prime (-s)
\mrm{cap}_\W (u< -s) ds \le 1 
% \int_\W-\chi( u)(dd^c u )^n\le 1
\}$
and $C_\chi(u):= \int_0^\infty s^n \chi^\prime (-s)
\mrm{cap}_\W (u< -s) ds.
% \int_\W -\chi(u)\left (dd^c u \right )^n
$
%denotes the $\chi$-energy of $u.$
\end{thm}
The equivalences $(1) \iff(3)\iff(4)$  are proved in \cite{B09} (Theorem 5.1) and the implication $(5) \Longrightarrow(1)$ 
is proved in 
\cite{BGZ2} (Theorem 5.2). In the sake of 
completeness we include a complete proof.
%The proof of this theorem can be found in \cite{B _3}. For the convenience of the reader,
% we include a complete proof.
\begin{proof}
%The proof of $(1)\Leftrightarrow (3)\Leftrightarrow (4)$ can be found in \cite{B 1}.
% The proof of $(1)\Leftrightarrow (2)$ has been done in \cite{BGZ 2}.
%  For he convenience of the reader, we include here a complete proof.
 We start by the implication $(1)\Rightarrow (2).$
 Let $u, \ \f\in {\m E _\chi}(\W).$
 It follows from Proposition \ref{capdef} that $u+\f \in \m E_\chi(W).$ Hence
$$
\int_\W -\chi(u)\left (dd^c \f \right )^n \le \int_\W -\chi(u+\f)\left (dd^c (u+\f ) \right )^n<\infty.
$$
Now, for the implication $(2) )\Rightarrow (3),$ assume that (3) is not satisfied. Then for each
$j\in \N$ we can find a function $u_j \in \tilde{\m E}_0 (\W) $ such that 
\begin{equation}
\label{eq11}
\int_\W-\chi(u_j) d \mu \ge 2^{3j} .
\end{equation}
Consider the function
$$
u:= \sum_{j=1}^\infty {1\over 2^{2j}} u_j .
$$
Observe that 
$$
(u<-s) \sub \cup _1 ^\infty \left ( u_j <-{2^j s} \right ).
$$
Hence
$$
\mrm{cap}_\W (u<-s) \le \sum_j ^\infty \mrm{cap}_\W(u_j < -2^j s).
$$
Now, since the weight $\chi $ is convex or homogeneous and using the estimates (\ref{estimate}),
we get 
\begin{eqnarray*}
\int_0^\infty s^n \chi^\prime (-s)
\mrm{cap}_\W (u< -s) ds & \le & \int_0^\infty s^n \chi^\prime (-s)
 \sum_j ^\infty \mrm{cap}_\W(u_j < -2^j s)ds\\
&\le & \sum_j ^\infty \frac{1}{2^{nj-n}} \int_0^\infty (2^{j-1}s)^n \chi^\prime (-s)  \mrm{cap}_\W(u_j < -2^j s)\\
&\le &  \sum_j ^\infty \frac{1}{2^{nj-n}} \int_0^\infty (2^{j-1}s)^n \chi^\prime (-s)  \mrm{cap}_\W(u_j < - s)\\
& \le & 2^n \sum_j ^\infty \frac{1}{2^{nj}} < \infty.
\end{eqnarray*}
Hence $u \in \m E_\chi (\W).$ On the other hend, from (\ref{eq11}) we have 
$$
\int_\W -\chi (u) d\mu \ge  \frac{1}{2^{2j}}\int_\W -\chi (u_j) d\mu \ge 2^j, \quad \forall j\in \N,
$$
which yields a contradiction.

%Now, we prove the implication $(3) )\Rightarrow (4).$

  Now, we prove that  $(3) \Ri (4)$. Let $\p \in \m E _0(\W),$ denote $E_\vk
(\p): = \int_\W -\chi (\p) (dd^c \p)^n.$
  If $\p \in
\tilde{\m E _0}(\W)$,   i.e.  $C_\vk (\p)  \le 1 $  then
$$
 \int_\W -\chi (\p) d\mu \le 2^n= C.
$$
If  $C_\vk (\p)  > 1 .$ The function $\tilde{\p}$ defined by
$$
\tilde{\p}:=\frac{\p}{1+ C_\vk (\p)^{1/n}} \in   \tilde{\m E _0}(\W).
$$
Indeed, from the monotonicity of $ \chi $, we have
\begin{multline*}
\int_0^\infty \chi^{\prime}(s) s^n \mrm{cap}_\W \left( \frac{\p}{1 + C_\vk (\p)^{1/n}} <-s \right )ds \\
= \frac{1}{C_\chi (\p)}
 \int_0^\infty \chi^{\prime}(s) (sC_\chi(\p)^{1/n})^n \mrm{cap}_\W \left( {\p} <-s -sC_\vk (\p)^{1/n}  \right )ds \\
\le 
 \frac{1}{C_\chi (\p)}
 \int_0^\infty \chi^{\prime}(s) s^n \mrm{cap}_\W \left( {\p} <- s \right )ds=1.
\end{multline*}
It follows from (\ref{quant}) and the convexity of $\chi$
 %(or quasi-homogeneity)
\begin{multline*}
\int_\W -\chi (\p) d\mu  \le
 {2 C_\vk (\p)^{1/n}} \int_\W
-\chi \left ( \frac{\p}{1 + C_\vk (\p)^{1/n}}\right ) d\mu  \le  A {C_\vk
(\p)^{1/n}}.
\end{multline*}
Hence we get (\ref{quant1}).% with $C_2 = \max(1, C_1).$

For the implication $(4) \Ri (5),$ we consider $F(t) = A \max(1, t^{1/n}).$

$(5) \Ri (1).$ It follows from \cite{BGZ2}(Theorem 4.5) that the class $\m E_\chi(\W )$
characterizes pluripolar sets in the sense that if $P$ is a locally pluripolar subset 
of $\W$ then $P \subset \{ v = - \infty\},$ for some $v \in \m E _\chi (\W)$.
 Then the assumption (\ref{quant2}) on $\mu $ implies 
 that it vanishes on pluripolar sets. It follows from \cite{Ce04}
 that there exists a function $u \in \m E _0(\W)$ and
 $f \in L_{loc}^1 \big((dd^c u )^n\big)$ such that
$\mu = f (dd^c u )^n. $

Consider $\mu_j:=\min (f, j ) (dd^c u )^n$. This is a finite measure
which is bounded from above by the complex Monge-Amp\`ere measure of
a bounded function. It follows therefore from \cite{K94} that there
exist $\f_j \in \m E _0 (\W)$
 such that
$$
(dd^c \f_j )^n = \min (f, j ) (dd^c u )^n.
$$
The comparison principle shows that $\f_j$ is a decreasing sequence.
Set  $\f =\lim_{j\to \infty } \f_j$. It follows from (\ref{quant2})
that
$$
\int_\W -\chi (\f_j) (dd^c \f_j)^n \le F
\left( \int_0^\infty  \chi^{\prime}(s) s^n \mrm{cap}_\W \left( \f_j <- s \right )ds   \right) .
$$
Hence
$$
  \sup_j    \int_0^\infty \chi^{\prime}(s) s^n \mrm{cap}_\W \left( \f_j <- s \right )ds
  <\infty.
$$
%So it follows from Proposition  \ref{capdef} that
%$$
%\sup_j \int_0^{+\infty} t^{n }\chi'(-t) \mrm {Cap}_\W(\{\f_j<-t\})
%dt<+\infty ,
%$$
 which implies that
  $$
  \int_0^{+\infty} t^{n }\chi'(-t) \mrm {cap}_\W(\{\f<-t\}) dt<+\infty
  .
$$
Then $\f \not \equiv -\infty$
  and  therefore  $\f \in \mE_\vk(\Omega) $.

 We conclude now by continuity of the complex
Monge-Amp\`ere operator along decreasing sequences that
 $(dd^c \f )^n = \mu.$ The uniqueness  of $\f $ follows from the
 comparison principle.
\end{proof}
%%%%%%%%%%%%%%%%%%%%%%%%%%%%%
%%%%%%%%%%%%%%%%%%%%%%%%%%%%%%%%%%%%%%%%%%%%%%%%%%%%%%%%%%%%%%%%%%%%%%%%%%%%%%%%%%%%%%%%%%%%%%%%%%%%%%%%%%%%%%%%%%%%%%%%%
%%%%%%%%%%%%%%%%%%%%%%%%%%%%%%%%%%%%%%%%%%%%%%%%%%%%%%%%%%%%%%%%%%%%%%%%%%%%%%%%%%%%%%%%%%%%%%%%%%%%%%%%%%%%%%%%%%%%%%%%%%%%
\section{The weighted energy class with boundary values}
Let $\chi : \R^- \to \R^- $ be a non-decreasing function and let $f\in \m M(\W)$ be a maximal psh function.
 We define the class $\m E_\chi(f)$ (resp. $\m N(f),\ $ $ \m F (f),$ $ \ \m N^a (f),\ $ $ \m F ^a(f)$)
 to be the class of psh functions $u$ such that
there exists a function $\f \in \m E_\chi (\W)$ (resp. $\m N,\  \m F, \ \m N^a ,\ \m F ^a$)  such that
$$
\f (z) + f(z) \le u (z) \le f(z),\qquad \forall z \in \W.
$$
%In the same way, we define $\m N(f)$ as the set of psh functions $u$.....
Later on, We will use  repeatedly the following
 well known comparison principle 
from \cite{BT1} 
as well as its generalizations to the class $\m N(f)$ (cf  \cite{ACCP} \cite{Ce08}).
\begin{thm}[\cite{ACCP} \cite{BT1} \cite{Ce08}] \label{comparison}
Let $ f \in \m E(\W)$ be a maximal function and $u, \ v\ \in \m N(f)$  be such that $(dd^c u)^n $
vanishes on all pluripolar sets in $\W.$ % and $(dd^c u)^n =(dd^c v)^n. $
 Then 
$$
\int _{(u<v)} (dd^c v)^n \le \int _{(u<v)} (dd^c u)^n .
$$
Furthermore if $ (dd^c u)^n=  (dd^c v)^n$ then $u=v.$
%%%%%%%%%%%%%%%%%%%%%%%%%%%%%%%%%
%Assume $f\in \m E(\W)$ is maximal function, $u \in \m N^a (f)$ and $v \in \m E(\W) $ such that $v \le f- \e.$ Then
%$$
%\int _{(u<v)} (dd^c v )^n \le \int _{(u<v)} (dd^c u )^n .
%$$
\end{thm}
%\begin{proof} See \cite{Ce}.
%\end{proof}
%\begin{cor}
%Let $u \in \m N^a (f)$  and $v \in \m E(\W) $ such that $v \le f$ and $ (dd^c u )^n \le (dd^c v )^n.$ Then  $u \ge v.$

%In particular, if $ (dd^c u )^n = (dd^c v )^n$
% with $u,v  \in \m N^a (f)$ then $u=v.$
%\end{cor}

The following lemma, which gives an estimate of the size of sub-level set in terms of the mass
of Monge-Ampère measure,  will be useful shortly.
\begin{lem}\label{est}
 Let $\chi : \R^- \to \R^- $ be a non-decreasing function such that $\chi(t) < 0, \ \forall t<0$ and
$f\in \m E $ a maximal function.
% $\f \in {\mathcal E}_\chi (f)$. 
Then for all $\f \in \m E_\chi (f) $
\begin{equation}\label{cap}
 t^n
Cap_\W(\f\ <-s-t+f) \leq \int_{(\f< -s+f)} (dd^c \f)^n, \quad  \forall s>0 \ \text{ and}\  t > 0.
%\leq s^n
%Cap_\W(\f < -s )
\end{equation}
\end{lem}
\begin{proof} Fix  $s, \ t>0$.
 Let $K \subset \{ \f  < f-s -t \} $ be a  compact subset. Then
\begin{multline*}
\mrm{cap}_\W (K) = \int _\W \left ( dd ^c u_K^* \right )^n =
 \int _{ ( \f  < f-s -t )}
 \left( dd ^c u_K^* \right )^n \\
= \int _{( \f  < f-s + tu_K^* )}
 \left( dd ^c u_K^* \right )^n \le  \frac{1}{t^n}\int _
{ \{ \f  < v \}}
 \left( dd ^c v  \right )^n ,
\end{multline*}
where $u_K^*$  is the relative extremal function of the compact $K$
and $ v:= f-s + tu_K^* .$ It follows from Theorem \ref{comparison}  that
\begin{multline*}
\frac{1}{t^n} \int _{ \{ \f  < v \}}
 \left( dd ^c v \right )^n = \frac{1}{t^n}\int _
{ \{ \f  < \max( \f , v) \}}
 \left( dd ^c \max ( \f , v)  \right )^n
  \le\\
% \frac{1}{t^n}\int _{ \{ \f  <  -s -t+ t(u_K + 1) \}}
 %\left( nd ^c \max(v, \f)  \right )^n =\\
 \frac{1}{t^n}\int _{ \{ \f  < \max ( \f,  v) \}}
 \left( dd ^c  \f  \right )^n =    \frac{1}{t^n}\int _{
 \{ \f  <  f-s  + t u_K  \}}
 \left( dd ^c  \f)  \right )^n  \le  \frac{1}{t^n}\int _{
 \{ \f  <  f-s   \}}
 \left( dd ^c  \f)  \right )^n .
\end{multline*}
Taking the supremum over all $K$'s yields the first inequality.
\end{proof}

%\section{Dirichlet problem with continuous boundary data}

\begin{pro}
Let $\chi : \R^-  \to \R^- $ be a
 increasing
 function. Then we have
\begin{multline*}
\m E_\chi(f)\sub\\
 \left \{
  u \in PSH(\W); \ u\le f, \ \int _0 ^{+\infty} s^n \chi^\prime (-s)
\mrm {Cap }_\W (u< f -2s) ds <+\infty
 \right \}.
\end{multline*}
In particular, if $\chi \not \equiv 0$, then $\mrm {Cap }_\W (u< f -s)  <+\infty$ 
for all $s>0$ and $u \in \m E_\chi(f).$
\end{pro}
\begin{proof}
Let $u \in \m E_\chi(f).$
Then there exists a function
$\f  \in \m E_\chi(\W)$ such that $\f +f \le u.$
Therefore $( u< f-s)\sub (\f < -s).$
It follows from Lemma \ref{est}
\begin{multline*}
\int _0 ^{+\infty} s^n \chi^\prime (-s)
\mrm {Cap }_\W (u< f -2s) ds
\le
\int _0 ^{+\infty} s^n \chi^\prime (-s)
\mrm {Cap }_\W (\f<  -2s) ds \\
 \le  \int _0 ^{+\infty} \chi^\prime (-s) \int _{(\f <-s)} (dd^c \f )^n ds
 =  \int _\W -\chi(\f) (dd^c \f )^n <\infty.
\end{multline*}
\end{proof}
\begin{thm}
Let $\chi : \R^-  \to \R^- $ be an
 increasing
 function which satisfies $\chi(-\infty) = -\infty$ 
and $f \in \m E$ a maximal function. Then if there exists a decreasing sequence $u_j \in \m E_0 (f) $ 
such that
$$
\sup_j \int_\W -\chi(u_j - f )(dd^c u_j ) ^n < \infty
$$

then $u:= \lim_{j\to \infty} u_j  \in \m E _\chi (f) $ and $\chi(u - f ) \in L^1 ((dd^c u)^n).$

Conversely, if $u\in \m E _\chi (f) $ and $\chi(u - f ) \in L^1 ((dd^c u)^n)$ then there exists 
sequence $u_j \in \m E_0 (f) $ decreasing towards $u$ 
such that
$$
\sup_j \int_\W -\chi(u_j - f )(dd^c u_j ) ^n < \infty.
$$
\end{thm}
\begin{proof}
Assume that the sequence $u_j \in \m E_0(f) \cap C(\W)$ (if necessary,
 we approximate $u_j$ by a continuous sequence
$u_j^k \in \m E_0(f)$). For a fixed $j \in \N,$ let denote $\f_j $ the function defined by 
\begin{multline*}
\f_j (z) : = \sup \left \{ v \in PSH(\W) ; \ v + f \le u_j  \right \}, \quad \forall  z\in \W.
\end{multline*}
We claim that

1)$ \f_j \in \m E_0$;

2) $(dd^c \f_j)^n \le (dd^c u_j )^n;$ 

3) $(dd^c \f_j )^n = 0 $ on the subset  $ ( \f_j + f < u_j).$\\
Then, it follows from the statements 1), 2) and 3) that   for each $j \in \N $
\begin{multline*}
\int _\W -\chi(\f_j ) ^n (dd^c \f_j )^n  = \int _{(\f _j + f = u_j)} -\chi(\f_j ) ^n (dd^c \f_j )^n\\
\le \int _\W -\chi(u_j - f  ) ^n (dd^c u_j  )^n. 
\end{multline*}
Therefore
$$
\sup _j \int _\W -\chi(\f_j ) ^n (dd^c \f_j )^n \le 
 \sup_j  \int _\W -\chi(u_j - f  ) ^n (dd^c u_j  )^n < +\infty.
$$
Hence, the function $\f : = \lim _{j\to \infty}  \f_j $ satisfies $\f + f \le u \le f.$

For the converse implication, fix $u \in \m E _ \chi (f).$
Then there exists a function $\f \in \m E_\chi $ 
such that $\f + f \le u \le f.$
Let $ \f _j\in  \m E_0 \cap C(\W) $ be a decreasing  sequence with limit the function $\f.$ 
Then for each $ j \in  \N,$ consider the function 
$u_j : = \max (\f _j + f , u) \in \m E _ \chi (f).$
The sequence $u_j $ decreases towards $u$ and 
$$
 \int _\W -\chi(u_j - f  ) ^n (dd^c u_j  )^n  \le 
C \int _\W -\chi(u - f  ) ^n (dd^c u  )^n
< \infty,
$$
where $C$ is a constant which depends only on $u$ and the proof 
of the theorem is completed.
\end{proof}

\begin{thm}\label{slim}
Let $\mu $ be a non-negative measure in $\W ,$   $\chi : \ \R^- \to \R^- $ be
an increasing convex (or homogeneous) function such that $\chi(-\infty)
 = -\infty $ and $f \in \m E(\W)$ be a maximal function.
Then
 there exists  a unique function $\f \in \m E_\chi(f)$ such
that $\mu = (dd^c \f )^n$ if and only if $\mu $ satisfies one of the conditions of Theorem \ref{range}.
\end{thm}
\begin{proof}
Assume that
 $\mu = (dd^c v)^n$
 for some $v\in \mE_\chi.$ 
Let $(\W_j)_j$ be a fundamental sequence of strictly pseudoconvex subsets of $\W.$
Choose a sequence $f_j \in PSH(\W) \cap C(\bar{\W})$ decreasing towards $f$ on $\W$ and 
$f_j$ is maximal on $\W_{j+1}.$
It follows from \cite{Ce04} that there exist a function $g \in \m E_0$ and a function 
$\theta \in L^1_{loc}(dd^c g)^n $ such that
$$
\mu = \theta (dd^c g)^n .
$$
Consider the measure $\mu_j = \mathbb{1}_{\W_j} \min(\theta , j)  (dd^c g)^n,$
where $\mathbb{1}_{\W_j}$ denotes the characteristic function of the set $\W_j.$
Now, solving the Dirichlet problem in the strictly pseudoconvex domain $\W_j,$ 
we state that there exist
functions $u_j , \ v_j \in PSH(\W_j) \cap C(\bar{\W}_j)$ such that 
$$
(dd^c u_j )^n = (dd^c v_j )^n = \mu_j \quad \text{and} \quad v_j =0, \quad u_j = f_j \ \text{on} \ 
\partial \W_j.
$$
By the comparison principle, we have $u_j$  and $v_j$ are decreasing sequences and 
$$
v+ f \le v_j + f_j \le u_j \le f_j \quad \text{on}\quad \W_j.
$$
Letting $j \to +\infty $ we get that $u:= \lim_{j\to \infty} u_j \in \m E_\chi(f).$ The continuity 
of the complex Monge-Ampère operator under monotonic sequences yields that $(dd^c u)^n = \mu.$
Uniqueness of $u$ follows from the comparison principle.
\end{proof}
\begin{cor} \label{ahag}
Let $\mu $ be non-negative measure in $\W$ with total finite mass 
$\mu(\mu ) < \infty.$ and $f$ be a maximal function. Then there exists a uniquely determined function 
$\f \in \m F^a(f) $ such that $(dd^c \f )^n = \mu $ if and only if 
$\mu$ vanishes on pluripolar subsets.
\end{cor}
\begin{proof}
It follows from \cite{Ce04} that there exist a function $\p \in \m E_0$ and a function 
$\theta \in L^1_{loc}(dd^c \p)^n $ such that
$$
\mu = \theta (dd^c \p)^n .
$$
By \cite{K94}, there exists a unique $h_j \in \m E _0 $  such that $(dd^c h_j)^n = \min (\theta , j) (dd^c \p)^n.$
The comparison principle yields that $h_j$ is a decreasing sequence. Let denote by $h:= \lim_{j\to \infty} h_j.$
It follows from Lemme \ref{est} that $h\not \equiv -\infty.$ Therefore
$h \in \m F^a.$ By the continuity of the complex Monge-Ampère operator under decreasing sequences, we have 
$ (dd^c h)^n = \mu .$ Now, since 
$$
\m F^a = \bigcup_{\substack{ \chi\  \text{convex}; \chi(0) \not = 0\\
\chi(-\infty ) = -\infty}} \m E _\chi,
$$
then there exists a convex function $\chi : \R^- \to \R^- $ with $\chi(0) \not = 0$ and
$\chi(-\infty ) = -\infty $ such that $h \in \m E_\chi.$ By Theorem \ref{slim}, we can find a unique
 function $\f \in 
\m E_\chi (f) \sub \m F^a (f)$ such that $(dd^c \f)^n = \mu .$
\end{proof}

%\end{proof}
\section{Mesures dominated by Capacity}
Throughout this section, $\mu$ denotes a fixed non-negative  measure
of finite total mass $\mu(\Omega)<+\infty$. We want to solve the
Dirichlet problem
$$
(dd^c \f)^n=\mu ,\;
\text{ with } \f \in \m F^a (f)
\text{ and } \f_{|\partial \Omega}=f,
$$
and measure how far the distance between the  solution $\f$ and the given doundary data
$f$ is from being bounded, by
assuming that $\mu$ is suitable dominated by the Monge-Amp\`ere capacity.

Measures dominated by the Monge-Amp\`ere capacity have been extensively studied by
S.Kolodziej in \cite{K94}, \cite{K98} and \cite{K05}. The main result of his study, achieved
in \cite{K98}, can be formulated as follows. Fix $\e:\R \rightarrow  [0 ,
+\infty [$ a continuous decreasing  function and set
$F_{\e}(x):=x [\e(-\ln x/n)]^{n}$.
If for all compact subsets $K \subset \Omega$,
$$
\mu(K) \leq F_{\e}(\mrm{cap}_{\Omega}(K)), \text{ where } \int_0^{+\infty}
{\e(t)}dt <+\infty,
$$
and $l : \partial \W \to \R$ a continuous function, 
then $\mu=(dd^c \f)^n$ for some {\it continuous} function
$\f \in PSH(\Omega)$ with $\f_{|\partial \Omega}=l$.

\vskip.1cm

The condition $\int^{+\infty}_0 {\e(t)}dt  <+\infty$ means
that $\e$ decreases fast enough towards zero at infinity. This gives
a quantitative estimate on how fast $\e( -\ln Cap_{\Omega}(K)/n)$,
hence $\mu(K)$, decreases towards zero as $\mrm{cap}_{\Omega}(K) \rightarrow 0$.

When $\int^{+\infty}_0 \e(t)dt=+\infty$, it is still possible to
show that $\mu=(dd^c \f)^n$ for some function
$\f \in {\mathcal F}(\Omega)$, but $\f$ will generally be unbounded.
We now measure how far it is from being so:

\begin{thm} \label{kolo} Let $\mu $ be a non-negative finite measure.
Assume for all compact subsets $K  \subset \W$,
\begin{equation} \label{dom}
\mu(K) \leq F_{\e} \left (\mrm{Cap }_\W(K)\right ).
\end{equation}
Then there exists a unique function $\f \in \mF^a (f)$
such that $\mu=(dd^c \f)^n$,
and
\begin{equation*}
\mrm{Cap}_\W(\{\f<f-s  \}) \leq  \exp (-nH^{-1}(s)), \text{ for all }\ s>0,
\end{equation*}
Here $H^{-1}$ is the reciprocal function of
$H(x) =e\int_{0}^x  \e(t) dt +  e \e(0) + \mu(\W)^{1\over n} $.

%In particular $\f \in \mE _\chi (\Omega)$ with
%$-\chi(-t)=\exp (  n H^{-1}(t)/2)$.
\end{thm}

The proof is almost the same as that of Theorem 5.1 in \cite{BGZ2}, except that we use Corollary \ref{ahag}
for the existence of the solution and Lemma \ref{est} to estimate the capacity of sub-level set.

Observe that if $\int ^\infty \e (t) dt <\infty$ then $H$ is bounded by 
$e\int_{0}^\infty  \e(t) dt +  e \e(0) + \mu(\W)^{1\over n} .$ Hence $H^{-1} (t) = +\infty, \ \forall 
t \ge e\int_{0}^\infty  \e(t) dt +  e \e(0) + \mu(\W)^{1\over n} .$ Therefore
$$
0\le f- \f \le e\int_{0}^\infty  \e(t) dt +  e \e(0) + \mu(\W)^{{1\over n}}. 
$$

%\subsection*{Measures with density}
Now, we consider the case when $\mu = f d \lambda $ is absolutely 
continuous with respect to Lebesgue measure.

Let $\mG \subset \C ^n $ denotes a generic subspace of $\C^n :$ that is a real subspace such that 
$\mG + J \mG = \C ^n,$ where $J$ is the usual complex structure on $ \C^n$ (cf \cite{BJZ} for more details).
 $\mG $ will be endowed with the 
induced euclidean structure and the corresponding Lebesgue measure which will be denoted by $\lambda _\mG.$

Let $\a > 0$  be a positive real number. 
According to \cite{IM} and \cite{RR}, the Orlicz space  $L Log ^{n + \a} L(d\lambda _\mG)$ consists 
$\lambda _\mG$-measurable functions $g$ defined on $\W \cap \mG$ such that  
$$
 \int _{\W\cap \G}
 \frac{|f|}{\lambda} \log ^{n+\a } (1 + \frac{|f|}{\lambda}) d\lambda_\mG < \infty, 
\quad \text{for some }  \ \lambda >0
$$
On the space $L Log ^{n + \a} L(d\mu),$  we define the norm
$$
 ||f||_{L Log ^{n + \a} L} : = \inf \left \{ \lambda >0; \  \int _\W
 \frac{|f|}{\lambda} \log ^{n+\a } (e + \frac{|f|}{\lambda}) d\lambda_\mG <1 \right \}.
 $$
 The dual space to $L Log ^{n + \a} L,$ is the exponential class  $Exp L ^{1/n+\a}$:
  that is the vector space
 $$
 Exp L ^{1/n+\a} : = \left \{ f: \W \to \ove \R ; \ \exists \lambda >0 : \ \ 
\int _\W  \exp \left(\left (\frac{|f|}{\lambda }\right )^\frac{1}{n +\a } \right ) - 1  d\lambda_\mG
 < \infty \right \}
 $$
 equipped with the norm
 $$
 ||f||_{ Exp L ^{1/n+\a}}  : = \inf \left \{ \lambda >0; \ 
 \int _\W \left ( \exp \left( \left (\frac{|f|}{\lambda }\right )^\frac{1}{n +\a }\right ) - 1
 \right ) d\lambda_\mG
 <1 \right \}.
 $$
 Then we have the following H\"older inequality
 $$
 \left | \int _\W f g d \lambda_\mG \right |  \le C_{n, \a} ||f||_{L Log ^{n + \a} L} ||g||_{ Exp L ^{1/n+\a}},
 $$
  for $f \in L Log ^{n + \a} L $ and  $g \in Exp L ^{1/n+\a},$  where $C_{n, \a} >0$ is
 a positive constant depending only in $n$  and $\a .$
By a simple computation, we have 
\begin{equation}
\label{norm1}
||1||_{ Exp L ^{1/n+\a}(K)}
 = {1\over \log^{n +\a} \left ( 1 + {1\over \lambda_\mG(K)} \right)}.
\end{equation}
\begin{cor}\label{horl}
Let $\mu = \mathbb{1}_{\W \cap \mG} g \lambda _\mG$ be a measure with non-negative density $g \in L \log ^{n+\a} L (\W  \cap \mG).$ Then 
there exists a unique bounded function
$\f \in \m F^a(f) \cap L^\infty (\W) $ such that $(dd^c \f ) ^n = \mu $ and 
\begin{equation}\label{eq41}
0\le f - \f \le  C ||g||_{  L \log^{n+\a} L }^{1\over n },
\end{equation}
where $C>0 $ only depends on $n, \ \a, \ \W $  and $\mG.$

\end{cor}
\begin{proof}
We claim that there exists a constant $C>0$ such that 
\begin{equation}\label{eq42}
\mu (K) \le \left( C ||g||_{  L \log ^{n+\a}L}^{1\over n} \right ) ^n \mrm{cap}_\W  ^{\a + n \over  n} (K), \quad 
\text{for all compact  }\ K \sub \W.
\end{equation}
Indeed, H\"older's inequality and inequality (\ref{eq42}) yield 
\begin{equation}\label{eq43}
\mu (K) \le  ||g||_{  L \log ^{n+\a}L }{1\over \log^{n +\a} \left ( 1 + {1\over \lambda_\mG(K)} \right)} (K), \quad 
\text{for all compact }\ K \sub \W.
\end{equation}
By \cite{BJZ} we have
\begin{equation}\label{eq44}
\lambda_\mG (K) \le C \exp  \left (-{ 1\over \mrm{cap}^{1\over n}_\W(K)} \right),  \quad
\text{for all compact }\ K \sub \W,
\end{equation}
where $C>0$ is a constant which depends only on $\W $ and $ \mG .$ \\
The inequality (\ref{eq42}) follows by combining (\ref{eq43}) and (\ref{eq44}).

Then we apply Theorem \ref{kolo} with
$$
\e(x) =  C ||g||_{  L \log ^{n+\a}L} C ||g||_{  L \log ^{n+\a}L}^{1\over n}  e^{-{\a x \over n}},
$$
which yields to
$$
0\le f - \f \le e\int_{0}^x  \e(t) dt +  e \e(0) + \mu(\W)^{1\over n}
  \le  C ||g||_{  L \log^{n+\a} L }^{1\over n }.
$$
\end{proof}
 \subsection*{Acknowledgements}
The author wishes to thank the referee for his
careful reading and for his remarks which helped to improve the exposition.

%benel@math.ups-tlse.fr

\end{document}